\documentclass{amsart}

\usepackage{amsfonts,amsmath,amssymb}
\usepackage{xcolor}
\usepackage{latexsym,verbatim}
\usepackage{paralist}
\usepackage{ulem}
\usepackage{graphics}

\newtheorem{theo}{Theorem}[section]
\newtheorem{prop}[theo]{Proposition}
\newtheorem{coro}[theo]{Corollary}
\newtheorem{lema}[theo]{Lemma}

\theoremstyle{remark}
\newtheorem{rema}[theo]{Remark}

\title[Rational solutions of Abel differential equations]{Rational solutions and limit cycles of polynomial and trigonometric Abel equations}

\author{L.A. Calder\'on*}
\address{Departament de Ciencies Matematiques i Informatica, IAC3 Institute of Applied Computing \& Community Code, Universitat de les Illes Ballears, 07122 Palma, Spain}
\email{l.calderon@uib.es}


\begin{document}
	
	\begin{abstract}
		We study the Abel differential equation $x'=A(t)x^3+B(t)x^2+C(t)x$. Specifically, we find bounds on the number of its rational solutions when $A(t), B(t)$ and $C(t)$ are polynomials with real or complex coefficients; and on the number of rational limit cycles when $A(t), B(t)$ and $C(t)$ are trigonometric polynomials with real coefficients.
	\end{abstract}
	
	\maketitle
	
	\section{Introduction}
	
	The Abel differential equation
	\begin{equation}\label{intro1}
		x'=A(t)x^3+B(t)x^2+C(t)x
	\end{equation}
	with $A(t)$, $B(t)$, $C(t)$ continuous, and the generalized Abel equations
	\begin{equation}\label{intro2}
		x'=\displaystyle\sum_{i=1}^nA_i(t)x^i
	\end{equation}
	with $n>3$ and $A_i(t)$ continuous for all $i\in\{1,\dots,n\}$ have been widely studied in the context of Hilbert's 16th Problem, and because of their utility to model real-world models \cite{BNP} and its own mathematical interest \cite{G}. 
	
	In the field of the qualitative theory of differential equations, the two most studied problems regarding Abel equations are the Smale-Pugh problem and the Poincaré Centre-Focus problem. 
	
	The Smale-Pugh problem \cite{S} consists on bounding the number of limit cycles (isolated periodic solutions) of \eqref{intro1} when its coefficients with respect to $t$ are periodic. Lins Neto \cite{LN} showed that there is no bound for the number of limit cycles of \eqref{intro1} when $A(t)$, $B(t)$ and $C(t)$ are trigonometric polynomials. This problem is related to Hilbert's 16th problem through the transformation of Cherkas \cite{CH}. Some results in this line can be consulted in, for example, \cite{ABF,ABF2,AGG,GG,GLL,HL1,HL2,HZ,I,LL,Pan,P,YHL}.
	
	The other problem is the adaptation of the classical Poincaré Centre-Focus problem to this setting, proposed by Briskin, Françoise and Yondim \cite{BFY1,BFY2}. Here, as $x(t)\equiv0$ is always a periodic solution of \eqref{intro1}, we will say that the equation has a center if every solution in a neighbourhood of $x(t)\equiv0$ is also periodic. For more details on the state of art of this problem, see for example \cite{G,GGS,Pa}.
	
	Another natural question regarding an Abel equation is related to how many solutions of a certain type the equation can have. For example, Giné et al \cite{GGLL} showed that the generalized Abel equations of degree $n$ in $x$ with polynomial coefficients have at most $n$ polynomial solutions. Other works about similar topics are \cite{CGM,GTZ,LLV2,LLWW,OV,QSY,V3}. 
	
	In \cite{LLV,V2}, the authors study how many rational closed solutions can the Abel equation 
	\begin{equation}\label{Abelsinc}
		x'=A(t)x^3+B(t)x^2
	\end{equation}
	have, both in the polynomial (this is, when $A(t)$ and $B(t)$ are polynomials and the rational solutions are quotients of polynomials) and in the trigonometric cases (this is, when $A(t)$ and $B(t)$ are trigonometric polynomials, and the rational solutions are quotients of trigonometric polynomials). Note that, in the trigonometric case, this problem is related with the Smale-Pugh problem, as, if the equation has not a center, the rational solutions are periodic and, consequently, limit cycles of the equation. Thus, the number of certain types of limit cycles of the equation is being determined.
	
	In \cite{BCFO}, we studied the polynomial case for the equation \eqref{Abelsinc}, for polynomials with real or complex coefficients, but looking for all rational solutions of the equation, regardless of whether they are closed or not. In particular, conditions are given for the equation to have at most two rational solutions; a general bound for the number of rational solutions is given; and Darboux's Integrability Theory (see more in \cite{GGLL2} for example) is used to determine how many rational solutions can coexist before the equation is integrable.
	
	Furthermore, in \cite{BCO} the trigonometric case of the equation \eqref{Abelsinc} is studied and bounds are found for the number of rational limit cycles, using Darboux's Integrability Theory to improve the bounds given in \cite{V2}. This work is only performed in the real context, since the techniques developed in the polynomial case do not translate adequately in the complex environment. 
	
	Clearly, there are various ways to extend this problem, for example, by searching results about the number of rational solutions of generalized Abel equations. However, the most immediate extension consists on adding the linear term to the equation, this is, to look for rational solutions of
	\begin{equation}\label{ecu0}
		x'=A(t)x^3+B(t)x^2+C(t)x+D(t).
	\end{equation}
	Without loss of generality, we assume $D(t) \equiv 0$. This problem has been studied in \cite{V}, but only for closed rational solutions in the polynomial case. 
	
	The objective of this work is to extend the techniques and results in \cite{BCFO,BCO} to the ``complete'' Abel differential equation.
	
	In precise terms, throughout this paper, we consider the Abel equation
	\begin{equation}\label{ecu1}
		x' = A(t) x^3 + B(t) x^2 + C(t) x
	\end{equation}
	with $A(t)$, $B(t)$, and $C(t)\in R$, where $R$ is either the polynomial ring in $t$ with coefficients in $\mathbb{K} \in \{\mathbb{R}, \mathbb{C}\}$, in symbols, $R = \mathbb{K}[t]$, or the ring of trigonometric polynomials with real coefficients, in symbols, $R = \mathbb{R}[\cos(t), \sin(t)]$. Note that we exclude the study over $\mathbb{C}[\cos(t), \sin(t)]$ because, as mentioned above, the techniques can not be applied appropriately in this setting. 
	
	Moreover, since the case $C(t) \equiv 0$ is completely studied in \cite{BCFO,BCO}, in what follows we consider $C(t) \not\equiv 0$, and we also assume that $A(t) \not\equiv 0$ and $B(t) \not\equiv 0$, since otherwise \eqref{ecu1} would be a Ricatti or Bernoulli equation, respectively. We will also assume that $A(t)$ is non constant, as this case has been studied by Pliss in \cite{P}, in the trigonometric case, and by Gine, Grau and Llibre \cite{GGLL} in the polynomial case (see Remark \ref{RemAconstant}).
	
	%
	
	Whenever possible, we will treat both the polynomial case and the trigonometric case as a whole. For this reason, it is convenient to note that $\deg(-)$ indistinguishable means in the polynomial case the usual degree of an univariate polynomial; while in the trigonometric case it means the highest integer $n$ such that $a_n \neq 0$ or $b_n \neq 0$ in the expression 
	\[a_0 + \sum_{n=1}^N a_n \cos(n t) + b_n \sin(n t).\] 	
	Moreover, to facilitate the cases integration, we recall that the problem of finding rational solutions of the equation \eqref{ecu1}, in the polynomial case, and rational solutions (that will be automatically periodic) of \eqref{ecu1}, in the trigonometric case, is equivalent to studying the invariant curves of \eqref{ecu1} of the form $p(t)x+q(t)=0$, with $p(t), q(t) \in R$ and $p(t) \neq 0$ for every $t$. Therefore, we use the term invariant curve of degree one in $x$, (or simply invariant curve, since there can be no confusion here) to unify the terminology, and the results will be given, always that is possible, in terms of invariant curves and not rational solutions. For more details on invariant curves see \cite{GGLL2}.
	
	We are now ready to present the main result in this paper. Note that it is written in terms of invariant curves, with the exception of the case (b.2.2), where we can improve the bound by studying instead the number of rational limit cycles, which is equivalent to the number of invariant curves when the equation does not have a center.
	
	\begin{theo}\label{mainth}
		With the above notation and conventions, the following claims are satisfied.
		\begin{itemize}
			\item[(a)] If $\deg(A)-\deg(C) \leq 1$, then \eqref{ecu1} has at most one invariant curve. 
			\item[(b)] If $\deg(A)-\deg(C) > 1$ and we are either in the polynomial case, or in the trigonometric one with $C(t)$ non-constant, then
			\begin{itemize}
				\item[(b.1)] If $\deg(A)-\deg(C)$ is odd or $\deg(A)+\deg(C) <  2 \deg(B)$, then \eqref{ecu1} has at most two invariant curves.
				\item[(b.2)] If $\deg(A)-deg(C)$ is even and $\deg(A)+\deg(C) \geq 2 \deg(B)$, then,
				\begin{itemize}
					\item[(b.2.1)]  in the polynomial case, the number of invariant curves of \eqref{ecu1} is at most \[\binom{\deg(A)}{(\deg(A)+1)/2}+1.\] Moreover, in this case, \eqref{ecu1} is Darboux integrable, if the number of invariant curves of \eqref{ecu1} is greater that $\deg(A)+\deg(C)+4$.
					\item[(b.2.2)] in the trigonometric case, an upper bound for the number of rational limit cycles of \eqref{ecu1} is $2\deg(A)+2\deg(C)+4$.
				\end{itemize}
			\end{itemize}
			\item[(c)] If $\deg(A)-\deg(C) > 1$, in the trigonometric case and with $C(t)\equiv c$ a non zero constant, then an upper bound for the number of invariant curves is  $4\deg(A)$.
		\end{itemize}
	\end{theo}
	
	Note that we have assumed that $C(t)\not\equiv 0$, and so this result does not cover the results of \cite{BCFO,BCO}. However, the conditions of (b.1) are very similar to their equivalent versions in \cite{BCFO,BCO} for the case $C(t)\equiv 0$, and the general bound of (b.2.1) for the polynomial case is exactly the same. On the other hand, when we apply Darboux's Theory of Integrability in this context, because of details of the proof the bounds of (b.2.2) and (c) are less sharp than their equivalent ones in \cite{BCFO,BCO}. 
	
	In \cite{BCFO} using computational numerical and algebraic tools we were able to develop a method to study thoroughly the number of invariant curves that equation \eqref{Abelsinc} has, in the polynomial case. Its computational effort grows sharply as $\deg(A)$ rises though, so we only studied the cases up to $\deg(A)=5$.
	
	Using this method, in \cite{BCFO} for low values of $\deg(A)$ we could study the sharpness of the general bound \[\binom{\deg(A)}{(\deg(A)+1)/2}+1\] for the number of invariant curves of the equation \eqref{Abelsinc}. We did show that this bound was sharp when $\deg(A)$ was one or three, but it was not sharp when it was five. 
	
	On the other hand, we did not study the sharpness of the bounds provided by Darboux's Theory of Integrability when $C(t)\equiv 0$ in \cite{BCFO,BCO}, although some numerical experiments seemed to show that they were sharp for low values of $\deg(A)$. In general, the sharpness of these bounds is an open question. 
	
	As this paper covers the polynomial and the trigonometric case at the same time, we have decided to not develop an equivalent exhaustive method for $\eqref{ecu1}$, although we offer a lemma related with the parametrization of the equation in terms of its invariant curves, which is very helpful to make numerical experiments in this direction. Consequently, the sharpness of the bounds provided in Theorem \ref{mainth} are left as an open question. 
	
	The structure of the paper is as follows: first we remind the factorization properties of the ring $\mathbb{R}[\cos t,\sin t]$ and in section 2.1 we show the characterization of the invariant curves of \eqref{ecu1}; then, in section 2.2, we suppose that \eqref{ecu1} has two or more different invariant curves, and we prove (a) and (b.1); next, in section 2.3, we focus in the polynomial case to provide the general bound of (b.2.1); and finally in section 3 we apply Darboux's Theory of Integrability to prove the remaining parts of (b.2.1), together with (b.2.2) and (c).

	\section{Invariant curves of degree one in $x$}
	
	%
	
	We remind that we have set $R$ to be equal to $\mathbb{K}[t],\ \mathbb{K} \in \{\mathbb{R}, \mathbb{C}\}$, or equal to $\mathbb{R}[\cos(t), \sin(t)]$. While in the first case $R$ is a Euclidean domain and everything works fine in terms of divisibility or factorization of elements, in the second case $R$ is not even a unique factorization domain. So special care must be taken when using it with factorization arguments. Nevertheless, $\mathbb{R}[\cos(t), \sin(t)]$ has somewhat stable factorization properties for being a half-factorial Dedekind domain. We summarize those extensively used along this paper and refer the reader to \cite[Appendix A]{BCO} for more details.
	
	\begin{prop}\label{Prop trigpoly}
		Let $p(t)\in\mathbb{R}[\cos(t),\sin(t)]$. If $p(t)\notin\mathbb{R}$ and $p(t)\neq0$ for every $t$, then the following holds.
		\begin{itemize}
			\item[(a)] The decomposition of $p(t)$ into irreducibles of $\mathbb{R}[\cos(t),\sin(t)]$ is unique, up to the order of the factors and the product by invertible elements.
			\item[(b)] If $p(t) = h(t) r(t)$ and none of the irreducible factors of $p(t)$ divides $h(t)$, then $r(t)$ divides $p(t)$. 
			\item[(c)] If $q(t)\in\mathbb{R}[\cos(t),\sin(t)] \setminus \{0\}$ is non-constant, then $p(t)$ and $q(t)$ are relatively prime if and only if they have not irreducible factors in common. 
		\end{itemize}
	\end{prop}
	
	As mentioned before, we are interested in studying the invariant curves of degree one in $x$ (that from now on we call simply invariant curves) of \eqref{ecu1}, that is, the invariant curves \eqref{ecu1} of the form $p(t)x+q(t)=0$, with $p(t), q(t) \in R$, 
	$p(t)$ non-constant and $p(t) \neq 0$ for every $t$. Because $p(t) \neq 0$ we always assume that $p(t)$ and $q(t)$ are relatively prime.
	
	Recall that $p(t)x+q(t)=0$ is invariant for \eqref{ecu1} if and only if there exists $K(t,x) \in R[x]$, called cofactor of $p(t)x+q(t)$, such that 
	\begin{equation}\label{ecu:inv}
		p'(t)x+q'(t)+p(t)\left(A(t)x^3+B(t)x^2+C(t)x\right)=(p(t)x+q(t))K(t,x).
	\end{equation}
	For degree reasons, $K(t,x) = K_2(t) x^2 + K_1(t) x + K_0(x)$ with $K_i(t) \in R,\ i = 0, 1, 2$. 
	
	\subsection{One invariant curve}
	A characterization of the invariant curves of the equation \eqref{Abelsinc} can be found in \cite{LLWW}, which was then adapted to our case, and the trigonometric case with $C(t)\equiv0$ in \cite{V,V2}, respectively. We include its proof here for the sake of completeness.
	
	\begin{prop}\label{prop1}
		The curve $p(t)x+q(t)=0$ is invariant for \eqref{ecu1} if and only if $q(t)=c $ is a nonzero constant and 
		\begin{equation}\label{ecu2}
			c^2 A(t ) = \left(c B(t) - \left(p'(t)+p(t)C(t)\right)\right)\, p(t).
		\end{equation}
	\end{prop}
	
	\begin{proof}
		%
		By \eqref{ecu:inv}, $p(t)x+q(t)=0$ is invariant for \eqref{ecu1} if and only if \[ p'(t)x+q'(t)+p(t)\left(A(t)x^3+B(t)x^2+C(t)x\right)=\] \[=(p(t)x+q(t))\left(K_2(t)x^2+K_1(t)x+K_0(t)\right),\] with $K_i(t) \in R,\ i = 0, 1, 2$. So, equating the coefficients of the powers of $x$ and omitting the arguments, we have
		\[q'=q\, K_0,\]
		\[p'+p\, C=K_0\,p + K_1\, q,\]
		\[p\, B = K_1\, p+K_2\, q,\]
		\[p\, A = K_2\, p \Longleftrightarrow A = K_2\ (\text{since}\ p(t)\not\equiv 0).\]
		From the first equality, we obtain both $K_0(t)\equiv0$ and $q(t)=c$ for a nonzero constant $c$. Now, substituting this conveniently into the second equation, we get  
		\[p'(t)+p(t)C(t) = K_1(t) c,\] that is, $K_1(t) = (p'(t)+p(t)C(t))/c.$ So the third equation yields 
		\[p(t)B(t)=(p'(t)+p(t)C(t))p(t)/c + K_2(t) c.\] Therefore, 
		$K_2(t) = p(t)B(t)/c - (p'(t)+p(t)C(t))p(t)/c^2$ and, consequently, \[c^2\, A(t) = \left(c B(t) - \left(p'(t)+p(t)C(t)\right)\right)\, p(t),\] as claimed.
	\end{proof}
	
	Note that, without loss of generality, we can assume $c=-1$ in Proposition \ref{prop1}. We will assume this from now until the end of the paper. Thus, our invariant curves will always be of the form $p(t) x - 1 = 0$.
	
	Although the following result is a direct consequence of Proposition \ref{prop1} and its proof, we state it because we will use it extensively throughout the paper.
	
	\begin{coro}\label{cor1}
		If $p(t)x-1 = 0$ is an invariant curve of \eqref{ecu1}, then 
		$p(t)$ divides $A(t)$. In particular, 
		\begin{equation}\label{ecu3}
			-B(t) = A(t)/p(t) + p'(t) + p(t) C(t)
		\end{equation} and the cofactor of $p(t)x-1=0$ is equal to $A(t) x^2 -\left(p'(t) + p(t) C(t)\right)x$. 
	\end{coro}

	\begin{rema}\label{RemAconstant}
		As mentioned in the introduction, in the trigonometric case, the case where $A(t)$ is constant is extensively studied in \cite{P}. In the other case, since, by Corollary \ref{cor1}, $A(t)$ being constant implies that $p(t)$ is also constant (in fact, regardless of whether we are in the polynomial or the trigonometric case), we have that the corresponding invariant curve is polynomial; a case that has been extensively studied in \cite{GGLL}. Thus, in the rest of the paper, we always assume that $\deg(A)>0$ 
	\end{rema}

	\subsection{Two or more invariant curves}
	Now suppose \eqref{ecu1} has at least two different invariant curves.
	
	\begin{prop}\label{prop2b}
		If $p_1(t)x-1=0$ and $p_2(t)x-1=0$ are two different invariant curves of \eqref{ecu1}, then 
		\begin{equation}\label{ecu4}
			A(t) =  C(t) p_1(t) p_2(t) + \frac{(p_2(t) - p_1(t))'p_1(t) p_2(t) }{p_2(t)-p_1(t)}.
		\end{equation}
		In particular, $\deg(A) - \deg(C) \leq \deg(p_1) + \deg(p_2)$.
	\end{prop}
	
	\begin{proof}
		By \eqref{ecu3},
		\[-A(t ) = \left(B(t) + \left(p_1'(t)+p_1(t)C(t)\right)\right)\, p_1(t).
		\]
		and
		\[-A(t ) = \left(B(t) + \left(p_2'(t)+p_2(t)C(t)\right)\right)\, p_2(t).
		\]
		If we multiply the first expression by $p_2(t)$ and the second expression by $p_1(t)$ and then take the difference between them, we get
		\[-(p_2(t) - p_1(t)) A(t) = (p'_1(t) - p'_2(t)) p_1(t) p_2(t) + (p_1(t)-p_2(t)) C(t) p_1(t) p_2(t).\] Now, since $p_1(t) \not\equiv p_2(t)$, \eqref{ecu4} follows.
		
		Finally, for the second part of the statement, it suffices to observe that by \eqref{ecu4} we have that 
		\begin{equation}\label{inqgrA}
			\begin{split}
				\deg(A) & \leq \max\left\{\deg(C(t) p_1(t) p_2(t)), \deg\left( \frac{(p_2(t) - p_1(t))'p_1(t) p_2(t) }{p_2(t)-p_1(t)} \right) \right\} =  \\
				& = \deg(C(t) p_1(t) p_2(t)) = \deg(C) + \deg(p_1) + \deg(p_2).
			\end{split}
		\end{equation}
		Therefore, $\deg(A) - \deg(C) \leq \deg(p_1) + \deg(p_2)$ as claimed. 
	\end{proof}
	
	\begin{prop}\label{prop2}
		Let $C(t)$ be non-constant. If $p_1(t)x-1=0$ and $p_2(t)x-1=0$ are two different invariant curves of \eqref{ecu1}, then \[\deg(A)  - \deg(C) = \deg(p_1) + \deg(p_2).\] In particular, $\deg(A) - \deg(C) > 1$.
	\end{prop}
	
	\begin{proof}
		Let \[r(t) := \frac{(p_2(t) - p_1(t))'p_1(t) p_2(t) }{p_2(t)-p_1(t)}.\] 
		By \eqref{ecu4}, $r(t)$ is equal to $A(t) - C(t) p_1(t) p_2(t)$. Therefore, $r(t) \in R$ and $\deg(r) = \deg(p_1)+\deg(p_2)-1$, in the polynomial case, and $\deg(r) = \deg(p_1)+\deg(p_2)$, in the trigonometric case.
		
		Now, since $C(t)$ is not constant, in particular, $\deg(C) > 0$ and, consequently, $\deg(Cp_1p_2) > \deg(p_1)+\deg(p_2) \geq \deg(r)$. Therefore, since by \eqref{inqgrA} $\deg(A) \leq \deg(C p_1 p_2)$ and $\deg(A - C p_1 p_2) = \deg(r) < \deg(C p_1 p_2)$, we conclude that $\deg(A) = \deg(C p_1 p_2)$, that is, $\deg(A) - \deg(C) = \deg(p_1) + \deg(p_2)$, as claimed.
		
		The last statement is as simple as noting that $p_1(t)$ and $p_2(t)$ are non-constant, and so $\deg(p_1) + \deg(p_2)>1$. 
		%
		%
	\end{proof}
	
	\begin{rema}\label{rema1}
		Note that, in the polynomial case, $C(t)$ can be chosen constant in Proposition \ref{prop2} and everything remains the same because, in this case, $\deg(r) = \deg(p_1) + \deg(p_2) - 1 < \deg(C p_1 p_2)$. 
	\end{rema}
	
	\begin{rema}
		Observe that, in the polynomial case, \eqref{ecu4} can be interpreted in terms of Euclidean divisibility as follows: if $p_1(t)x-1=0$ and $p_2(t)x-1=0$ are two different invariant curves of \eqref{ecu1}, then the quotient and the remainder of the division of $A(t)$ by $p_1(t) p_2(t)$ are \[C(t)\ \quad{and}\quad r(t) = \frac{(p_2(t) - p_1(t))'p_1(t) p_2(t) }{p_2(t)-p_1(t)},\] respectively. Giving rise to a computable necessary condition for two factors of $A(t)$ to define invariant curves.
	\end{rema}
	
	Let us derive some consequences from Proposition \ref{prop2}. 
	
	\begin{coro}\label{coram1}
		If $\deg(A) - \deg(C) \leq 1$, then \eqref{ecu1} has at most one invariant curve. 
	\end{coro}
	
	\begin{proof}
		If $C(t)$ is not constant, then, by Proposition \ref{prop2}, we are done. This also extends to the polynomial case, when $C(t)$ is constant (see Remark \ref{rema1}). Thus, suppose $A(t)$ and $B(t)$ are non-constant real trigonometric polynomials and $C(t) = c \in \mathbb{R} \setminus \{0\}$; in particular, $\deg(A) = 1$ by hypothesis. Therefore, $p_1(t) = a\, A(t)$ and $p_2(t) = b\, A(t)$ for some $a, b \in \mathbb{R}\setminus\{0\}$ with $a \neq b$. Thus, by \eqref{ecu4}, 
		\[
		A(t) =  a b A(t) (c A(t) + A'(t)).
		\]
		Therefore, $c A(t) + A'(t)=\frac{1}{ab}$. Now, since $A(t) = a_0 + a_1 \cos(t) + b_1 \sin(t),$ for some $a_0, a_1$ and $b_1 \in \mathbb{R}$, we obtain that 
		$$c a_0 + (c a_1 + b_1) \cos(t) + (c b_1 - a_1) \sin(t) = \frac{1}{ab}.$$
		Therefore, 
		$$ca_0ab=1,\quad ca_1+b_1=0,\quad cb_1-a_1=0.$$
		From the last two equalities, $a_1=b_1=0$, in contradiction with $A(t)$ being non constant. 
	\end{proof}

	\begin{lema}\label{lema2}
		Let $p_1(t) x - 1 = 0$ and $p_2(t) x - 1 = 0$ be two different invariant curves of \eqref{ecu1}. If $\deg(p_1) \neq \deg(p_2)$, and 
		\begin{itemize}
			\item[(a)] $C(t)$ is not constant, then has no more invariant curves.
			\item[(b)] $C(t)$ is constant, then $\deg(A) = \deg(p_1) + \deg(p_2)$.
		\end{itemize}
	\end{lema}
	
	\begin{proof}
		(a) Suppose on contrary, that $h(t,x) = 0$ is an invariant curve of \eqref{ecu1} different from $p_1(t) x - 1 = 0$ and $p_2(t) x - 1 = 0$. By Proposition \ref{prop1}, there exists a polynomial $p_3(t) \in R$ such that $h(t,x) = p_3(t)x-1$. Since, by Proposition \ref{prop2}, $\deg(p_i) + \deg(p_j) = \deg(A) - \deg(C)$ for every $1 \leq i < j \leq 3$, we conclude that $\deg(p_i) = \deg(p_j) = \frac{\deg(A) - \deg(C)}2$ for every $1 \leq i < j \leq 3,$ in contradiction with the hypothesis.
		
		(b) Note that by Remark 2.7 the result is immediate in the polynomial case, so we focus in the trigonometric one. Let $c \in \mathbb{R} \setminus \{0\}$ be such that $C(t) = c$. Then, by \eqref{ecu3}, 
		\[
		-B(t) = A(t)/p_1(t) + p_1'(t) + p_1(t) c = A(t)/p_2(t) + p_2'(t) + p_2(t) c.
		\]
		Suppose that $\deg(A) \neq \deg(p_1) + \deg(p_2)$. Then, by Proposition \ref{prop2b}, $\deg(A) < \deg(p_1) + \deg(p_2)$. Therefore $\deg(A/p_1) < \deg(p_2)$ and $\deg(A/p_2) < \deg(p_1)$. 
		
		Now, since $\deg(p_1) \neq \deg(p_2)$, we can assume $\deg(p_1) < \deg(p_2)$, then the higher degree terms of $p'_2(t)$ and $p_2(t) c$ must cancel each other. Thus, if $p_2(t) = a_n \cos(n t) + b_n \sin(n t) +$ \textit{lower degree terms}, where $a_n$ and $b_n$ are not both zero, then 
		\[c a_n \cos(n t) + c b_n \sin(n t) = - n a_n \sin(t) + n b_n \cos(t).\] Therefore, $c a_n - n b_n = c b_n + n a_n = 0$ and, since $c \in \mathbb{R} \setminus \{0\}$, we obtain $a_n = b_n = 0$, a contradiction. Hence, $\deg(p_1) = \deg(p_2)$.
	\end{proof}
	
	\begin{rema}\label{rema2}
		
		Note that according to Remark \ref{rema1} nothing changes if $C(t)$ is constant in (a) of the previous lemma in the polynomial case.
	\end{rema}
	
	Observe that as an immediate consequence of the proof of the above result we have the following corollary. 
	
	\begin{coro}\label{cor3}
		
		Suppose that we are in the polynomial case or in the trigonometric case when $C(t)$ is not constant. If \eqref{ecu1} has three invariant curves, say, $p_i(t) x - 1 = 0,\ i = 1,2,3$, then \[\deg(p_i) = \frac{\deg(A) - \deg(C)}2,\] for every $i \in \{1,2,3\}$.
	\end{coro}
	
	\begin{prop}\label{propc}
		
		Suppose that we are in the polynomial case or in the trigonometric case when $C(t)$ is not constant. If $\deg(A)-\deg(C)$ is odd or $2 \deg(B)>\deg(A)+\deg(C)$, then \eqref{ecu1} has at most two invariant curves. 
	\end{prop}
	
	\begin{proof}
		Corollary \ref{cor3} clearly implies that if $\deg(A)-\deg(C)$ is odd, then \eqref{ecu1} can have at most two invariant curves.
		
		Suppose that $2 \deg(B)>\deg(A)+\deg(C)$. By \eqref{ecu3}, if $p_1(t)x-1=0$ is an invariant curve, then \[-B(t)=A(t)/p_1(t)+p_1'(t)+p_1(t)C(t),\] and thus, $\deg(B)\leq\max\{\deg(A)-\deg(p_1),\deg(p_1),\deg(p_1)+\deg(C)\}$. 
		Now, if \eqref{ecu1} has more than 
		two invariant curves, by Corollary \ref{cor3}, $\deg(p_1)=\frac{\deg(A)-\deg(C)}{2}$, and we conclude that $$\deg(B)\leq\max\left\{\frac{\deg(A)+\deg(C)}{2},\frac{\deg(A)-\deg(C)}{2}\right\}=\frac{\deg(A)+\deg(C)}{2}$$ getting a contradiction. 
	\end{proof}
	
	Let us parameterize the family of equations \eqref{ecu1} with at least two invariant curves, $p_1(t)x-1=0$ and $p_2(t)x-1=0$.
	
	First, we will introduce some notation. Let $q(t) = \gcd(p_1(t), p_2(t))$, this is also well defined in the trigonometric case by Proposition \ref{Prop trigpoly}(c), and let $s_1(t)$, $s_2(t)$ and $s(t)$ be the polynomials (resp. trigonometric polynomials) such that
	\begin{equation}\label{ecmcd}
		p_1(t)=q(t)s_1(t),\quad p_2(t)=q(t)s_2(t),\quad A(t)=q(t)s_1(t)s_2(t)s(t).\end{equation}
	Note that $\gcd(s_1(t), s_2(t)) = 1$.
	
	\begin{lema}
		Let $p_1(t)$ and $p_2(t)$ be two non-constant polynomials in $t$ with coefficients in $\mathbb{K}$ (resp. two non-constant trigonometric polynomials with real coefficients) such that $p_i(t) \neq 0$, for every $t \in \mathbb{R}$ and $i \in \{1,2\}$. With the previous notation, if 
		\[\frac{(p_2(t) - p_1(t))'p_1(t) p_2(t) }{p_2(t)-p_1(t)}\]
		is a polynomial (resp. a trigonometric), then $s_2(t) = s_1(t) + k\, \widehat{q}(t)$, where $k$ is a nonzero constant and $\widehat{q}(t)$ is either constant or a product of not necessarily distinct irreducible monic factors of $q(t)$. As a consequence, if $p_1(t)x-1=0$ and $p_2(t)x-1=0$ are two invariant curves of \eqref{ecu1}, the claim must also be true. Furthermore, 
		$$s(t)=-q'(t)+q(t)\left(C(t)+\frac{\widehat{q}'(t)}{\widehat{q}(t)}\right).$$
	\end{lema}
	
	\begin{proof}
		By \eqref{ecmcd}, we have 
		\[\frac{(p_2(t) - p_1(t))'p_1(t) p_2(t) }{p_2(t)-p_1(t)} = q(t) s_1(t) s_2(t) \left(q'(t) + q(t) \frac{(s_2(t)-s_1(t))'}{s_2(t)-s_1(t)} \right).\] So, by hypothesis, \[q(t)^2 \frac{(s_2(t)-s_1(t))'}{s_2(t)-s_1(t)}\] is a polynomial, because there is no factor of $s_2(t)-s_1(t)$ dividing $s_1(t)s_2(t)$ as $\gcd(s_1(t), s_2(t))=1$.
		
		Let us decompose $s_2(t) - s_1(t) = q_1(t) q_2(t)$ where 
		all the factors of $q_2(t)$ divide $q(t)$ and no factor of $q_1(t)$ divides $q(t)$; in particular $\gcd(q_1(t), q_2(t)) = 1$. Then $q_1(t)$ divides $s'_2(t)-s'_1(t) = q'_1(t) q_2(t) + q_1(t) q'_2(t)$. So, it divides $q'_1(t) q_2(t)$ by Proposition \ref{Prop trigpoly}. Therefore, $q_1(t)$ must be a nonzero constant. Now, by simply taking $k = q_1(t)$ and $\widehat{q}(t) = q_2(t)$, we are done. 
		
		On the other hand, as by \eqref{ecu3},
		$$A(t)/p_1(t)+p_1'(t)+p_1(t)C(t)=-B(t)=A(t)/p_2(t)+p_2'(t)+p_2(t)C(t),$$
		then by \eqref{ecmcd},
		\[
		s_2(t)s(t)+q'(t)s_1(t)+q(t)s_1'(t)+q(t)s_1(t)C(t) = \] 
		\[s_1(t)s(t)+q'(t)s_2(t)+q(t)s_2'(t)+q(t)s_2(t)C(t)
		\]
		or, equivalently,
		$$q(t)((s_2'(t)-s_1'(t))+C(t)(s_2(t)-s_1(t)))=(s_2(t)-s_1(t))(s(t)+q'(t))$$
		and
		$$s(t)=-q'(t)+q(t)\left(C(t)+\frac{s_2'(t)-s_1'(t)}{s_2(t)-s_1(t)}\right).$$
		As $s_2(t)-s_1(t)= k\, \widehat{q}(t)$, the previous expression reduces to
		$$s(t)=-\left(q'(t)+q(t)\left(C(t)+\frac{\widehat{q}'(t)}{\widehat{q}(t)}\right)\right).$$
	\end{proof}

	The converse also holds by direct verification using Corollary \ref{cor1}. Therefore, we obtain a parameterization of all cases of equation \eqref{ecu1} having at least two invariant curves. 
	
	
	\subsection{Bounds for the invariant curves in the polynomial case}
	
	In the polynomial case, it is interesting to look for a general bound on the number of invariant curves of \eqref{ecu1}. We recall that this is not interesting in the trigonometric case, since we can get a much better bound on the number of rational limit cycles using the Darboux theory of integrability (see Section \ref{Sect3}). Thus, the rest of this section applies only to the polynomial case, and then we are in a Euclidean domain setting.
	
	Given that if $p(t)x-1=0$ is an invariant curve of \eqref{ecu1}, then $p(t)$ is a divisor of $A(t)$ in $\mathbb{C}[t]$, the number of divisors of $A(t)$ would be an upper bound for the number of such invariant curves, up to proportionality. To search for this bound, we must first study the invariant curves of the equation whose polynomials are proportional. 
	
	\begin{prop}\label{propor1}
		The curves $p(t)x-1=0$, $Kp(t)x-1=0$ with $K\in\mathbb{K}\setminus\{0,1\}$ are invariant curves of $\eqref{ecu1}$ if and only if
		$$A(t)=Kp(t)(p'(t)+p(t)C(t)),\quad B(t)=-(K+1)(p'(t)+p(t)C(t)).$$
	\end{prop}
	
	\begin{proof}
		As $p(t)x-1=0$ is an invariant curve of \eqref{ecu1}, then by Corollary \ref{cor1} there must exist a polynomial $q(t)$ such that
		$$A(t)=p(t)q(t),\quad -B(t)=p'(t)+q(t)+p(t)C(t).$$
		Moreover, as $Kp(t)x-1=0$ is also an invariant curve of \eqref{ecu1}, again by Corollary \ref{cor1} there must also exist a polynomial $\bar{q}(t)$ such that
		$$A(t)=K p(t)\bar{q}(t),\quad -B(t)=Kp'(t)+\bar{q}(t)+Kp(t)C(t).$$
		Thus, $q(t)=K\bar{q}(t)$ and, equating both expressions of $B(t)$, one gets that $\bar{q}(t)=p'(t)+p(t)C(t)$, and, 
		$$A(t)=Kp(t)(p'(t)+p(t)C(t)),\quad B(t)=-(K+1)(p'(t)+p(t)C(t)).$$
		Conversely, taking $q(t)=K(p'(t)+p(t)C(t))$, one has that  $A(t)=p(t)q(t)$, $B(t)=p'(t)+q(t)+p(t)C(t)$ and, consequently, that $p(t)x-1=0$ is invariant. Picking $q(t)=(p'(t)+p(t)C(t))$, one would similarly prove that $Kp(t)x-1=0$ is also invariant. 
	\end{proof}
	
	\begin{prop}\label{propor2}
		Equation \eqref{ecu1} has at most two invariant curves whose polynomials are proportional.
	\end{prop}
	\begin{proof}
		Because of Proposition \ref{propor1}, we know that if $p(t)x-1=0$ is an invariant curve of \eqref{ecu1}, then there must exists an unique $K\in\mathbb{C}\setminus\{0,1\}$ such that $Kp(t)x-1=0$, and 
		$$A(t)=Kp(t)(p'(t)+p(t)C(t)),\quad B(t)=-(K+1)(p'(t)+p(t)C(t)).$$
		This, if $K\neq-1$ implies that
		$$p(t)=-\frac{K+1}{K}\frac{A(t)}{B(t)}.$$
		Now, if there exists a different pair of invariant curves with proportional polynomials, say $\bar{p}(t)x-1=0$, $\bar{K}\bar{p}(t)x-1=0$, then if $\bar{K}\neq-1$ also
		$$\bar{p}(t)=-\frac{\bar{K}+1}{\bar{K}}\frac{A(t)}{B(t)}=\frac{\bar{K}+1}{\bar{K}}\frac{K}{K+1}p(t)$$
		and we are done. 
		
		Finally, note that if $p(t)x-1=0$ and $-p(t)x-1=0$ are both invariant curves of \eqref{ecu1}, then $B(t)\equiv0$ and we have excluded that case.
	\end{proof}
	
	\begin{prop}\label{propbound}
		If $\deg(A)-\deg(C)\leq1$, \eqref{ecu1} has at most one invariant curve. In another case, if $\deg(A)-\deg(C)$ is odd or $\deg(A)+\deg(C) < 2 \deg(B)$, then \eqref{ecu1} has at most two invariant curves. In another case, an upper bound for the number of invariant curves of $\eqref{ecu1}$ is $\displaystyle{\binom{\deg(A)}{(\deg(A)+1)/2}+1}$.
	\end{prop}
	
	\begin{proof}
		It is enough to remember that the number of divisors of $A(t)$, up to proportionality, is $\binom{\deg(A)}{(\deg(A)+1)/2}$ and that at most two invariant curves have proportional polynomials. 
	\end{proof}
	
	\section{Darboux's Theory of Integrability}\label{Sect3}
	
	Now let us take advantage of Darboux's Integrability Theory to obtain results related to the integrability of the equation, in the polynomial case, and upper bounds on the number of rational limit cycles, in the trigonometric case.
	
	We first remind the basic result of this theory from the original work of Darboux \cite{D}, adapted both to the polynomial and trigonometric settings. 
	
	\begin{prop}\label{Darboux}
		If \eqref{ecu1} has algebraic irreducible invariant curves $h_1(t,x)=0,\dots,$ $h_r(t,x)=0$ with respective cofactors $K_1(t,x),\dots, K_r(t,x)$ and there exist not all null numbers $\alpha_1,\dots,\alpha_r$ such that
		\[\sum_{i=1}^r\alpha_iK_i(t,x)=0,\]
		then
		\[h(t,x):=\displaystyle\prod_{i=1}^rh_i(t,x)^{\alpha_i}\]
		is a first integral of \eqref{ecu1}. In this case, we say that \eqref{ecu1} is a Darboux integrable.
	\end{prop}
	
	
	\begin{prop}\label{PropD}
		If \eqref{ecu1} has invariant curves $p_1(t)x-1=0,\dots,p_r(t)x-1=0$ with respective cofactors $K_1(t,x),\dots, K_r(t,x)$, and we write $K_0(t,x)=A(t)x^2+B(t)x+C(t)$ for the cofactor of the curve $x=0$, then there exist numbers $\alpha_0,\dots,\alpha_r$ not all zero such that $$\sum_{i=0}^r\alpha_iK_i(t,x)=0$$ if and only if \[\sum_{i=1}^r\alpha_iA(t)/p_i(t)=0\quad \text{and}\quad 0 = \alpha_0 = \sum_{i=1}^r \alpha_i.\]
	\end{prop}
	
	\begin{proof}
		First, we recall that, by Corollary \ref{cor1}, the invariant curves $p_i(t) x - 1 = 0$ have cofactors $K_i(t,x)=A(t)x^2-(p_i'(t)+p_i(t)C(t))x = A(t)x^2 + (B(t)+ A(t)/p_i(t))x,\ i=1,\dots,r$. Thus,
		$$\sum_{i=0}^r\alpha_iK_i(t,x)=\alpha_0K_0(t,x)+\sum_{i=1}^r\alpha_iK_i(t,x)=$$
		$$=\alpha_0 \left(A(t)x^2+B(t)x+C(t)\right)+\sum_{i=1}^r\alpha_i (A(t)x^2 + (B(t)+ A(t)/p_i(t))x) = $$
		$$=\left(A(t) \sum_{i=0}^r\alpha_i  \right) x^2+ \left(\alpha_0 B(t) + \sum_{i=1}^r\alpha_i (B(t)+ A(t)/p_i(t)) \right)x + \alpha_0 C(t),$$
		and our claim follows. 
	\end{proof}
	
	We can now prove the main result of the paper. 
	
	\begin{proof}[Proof of Theorem \ref{mainth}]
		The claims (a) and (b.1) have been proved in Corollaries \ref{coram1} and Proposition \ref{propc}, respectively. Moreover, the first part of claim (b.2.1) has been proven in Proposition \ref{propbound}.
		
		Let us now prove the rest of claim (b.2.1). By Proposition \ref{prop2} and Corollary \ref{cor3}, if \eqref{ecu1} has $r \geq 3$ invariant curves, say, $p_i(t) x -1 = 0,\ i = 1, \ldots, r$, then $\deg(p_i) = (\deg(A)-\deg(C))/2,$ for every $i \in \{1, \ldots, r\}$. Moreover, by Corollary \ref{cor1}, $p_i(t)$ divides $A(t)$, for every $i \in \{1, \ldots, r\}$. Therefore $A(t)/p_i(t)$ is a polynomial of degree $d := (\deg(A)+\deg(C))/2$, for every $i \in \{1, \ldots, r\}$. 
		
		Now, since the dimension of the $\mathbb{K}-$vector space of polynomials in $t$ of degree up to $d$ is equal to $d+1$, if $r \geq 2 (d+2)$, then there exist $\lambda_1, \ldots, \lambda_{d+2}$ and $\mu_{d+3}, \ldots, \mu_{r}$  such that \[\sum_{i=1}^{d+2} \lambda_i A(t)/p_i(t) = \sum_{i=d+3}^{r} \mu_i A(t)/p_i(t) =  0.\] Thus, taking \begin{align*}
			\alpha_0 & = 0,\\  \alpha_i & = \left(\sum_{j=d+3}^{r} \mu_j\right) \lambda_i,\ i = 1, \ldots, d+2,\\ \alpha_i & = -\left(\sum_{j=1}^{d+2} \lambda_j\right) \mu_i,\ i = d+3, \ldots, r,   
		\end{align*}
		we have that $0 = \alpha_0 = \sum_{i=1}^r \alpha_i$ and that $\sum_{i=1}^r \alpha_i A(t)/p_i(t) = 0$. So, by Proposition \ref{PropD} and as $2(d+2)=\deg(A)+\deg(C)+4$, we are done.
		
		We prove now (b.2.2). Following the same steps of the polynomial case, one gets that if equation \eqref{ecu1} has $r \geq 3$ invariant curves, say $p_i(t) x -1 = 0,\ i = 1, \ldots, r$, then $A(t)/p_i(t)$ is a trigonometric polynomial of degree $\hat{d} := (\deg(A)+\deg(C))/2$, for every $i \in \{1, \ldots, r\}$.
		
		We can repeat the rest of the argument of the polynomial case verbatim, and in view of the dimension of the $\mathbb{R}$-vector space of trigonometric polynomials in $t$ of the degree up to $\hat{d}$ is $2\hat{d}+1$, we conclude that if $r\geq2(2\hat{d}+2)$, \eqref{ecu1} has a Darboux first integral, and also a center. Thus, it can not have rational limit cycles. Consequently, a bound for the number of rational limit cycles is $2(2\hat{d}+2) =2\deg(A)+2\deg(C)+4$.
		
		Finally, we prove (c). On the one hand, if $p_i(t)x-1=0$, $i=1,\dots,r$, are invariant curves of $\eqref{ecu1}$, we know that $\deg(A/p_i)\leq\deg(A)-1$ for all $i\in\{1,\dots,r\}$, as $\deg(p_i)\geq1$. Thus, all the polynomials belong to the $\mathbb{R}$-vector space of trigonometric polynomials in $t$ of degree up to $\bar{d}=\deg(A)-1$, whose dimension is $2\bar{d}+1$. Thus, if $r\geq2(2\bar{d}+2)$, the equation would have a center. 
		
		However, on the other hand, the Abel equation
		$$x'=A(t)x^3+B(t)x^2+cx$$
		can never have a center, as its displacement function $d(x)$, whose isolated zeros correspond one to one with the limit cycles of the equation (see \cite{LL}), has derivative
		$$d'(x)=\exp\left(\int_0^{2\pi}(3A(t)x^2+2B(t)x+c)dt\right)-1.$$ 
		So $d'(0)=\exp(2\pi c)-1\neq0$, and $d(0)=0$ because $x(t)\equiv0$ is a solution of the equation. Therefore, $d(x)\not\equiv0$ and the equation can never have a center. 
		
		To conclude, in this case equation \eqref{ecu1} must have less that $2(2\bar{d}+2)=4\deg(A)$ invariant curves.
	\end{proof}
	
	\section*{Acknowledgements}
	
	The author is partially supported by grant number PID2023-151974NB-I00 funded
	by MCIN/AEI/10.13039/501100011033.\\
	
	I would like to thank professor Mª Jesús Álvarez Torres from Universitat de les Illes Balears, as this paper was developed as the answer to one of her thoughtful questions as a member of the examination committee of my doctoral thesis, and as a natural extension of the content of the thesis itself.

\end{document}